\documentclass[a4paper]{article}
\usepackage[T1]{fontenc}
\usepackage{mathptmx}
\usepackage{amsmath,amssymb,amsthm,amsbsy}
\usepackage{mathrsfs}
\usepackage{graphics,graphicx,color}
\usepackage{hyperref}
\usepackage{enumerate}
\usepackage{xcolor}
\usepackage{orcidlink}
\usepackage{setspace}
\usepackage[pagewise,displaymath, mathlines]{lineno} 
\newtheorem{lem}{Lemma}[section]
\newtheorem{thm}[lem]{Theorem}
\newtheorem{cor}[lem]{Corollary}
\newtheorem{prop}[lem]{Proposition}

\newcommand{\im}{\mathrm{im}}

\def\ker{\mathop{\rm ker}}

\def\im{\mathop{\rm im}}

\newcommand{\reg}{\mathrm{Reg}}

\begin{document}
\begin{center}
\textbf{\large{The regular part of transformation semigroups that preserve double direction equivalence relation}}\footnote{Part of this work was presented in the Workshop on General Algebra / Arbeitstagung Allgemeine Algebra (AAA102), University of Szeged, Hungary, June 24-26, 2022.}\\
\vspace{0.5 cm} Kritsada Sangkhanan \orcidlink{0000-0002-1909-7514}%\footnote{Corresponding author.} 
\end{center}

\begin{abstract}
	Let $T(X)$ be the full transformation semigroup on a set $X$ under the composition of functions. For any equivalence relation $E$ on $X$, define a subsemigroup $T_{E^*}(X)$ of $T(X)$ by
	\[
	T_{E^*}(X)=\{\alpha\in T(X):\text{for all}\ x,y\in X, (x,y)\in E\Leftrightarrow (x\alpha,y\alpha)\in E\}.
	\]
	In this paper, we show that the regular part of $T_{E^*}(X)$, denoted $\reg(T)$, is the largest regular subsemigroup of $T_{E^*}(X)$. Then its Green's relations and ideals are described. Moreover, we find the kernel of $\reg(T)$ which is a right group and can be written as a union of symmetric groups. Finally, we prove that every right group can be embedded in that kernel.
\end{abstract}

\section{Introduction}
Let $T(X)$ be the set of all functions from a set $X$ into itself. We have $T(X)$ under the composition of functions is a regular semigroup which is called the \textit{full transformation semigroup on $X$}. Nowadays, this semigroup plays an important role in algebraic semigroup theory. Analogous to Cayley's Theorem for groups, we can show that every semigroup $S$ can be embedded in the full transformation semigroup $T(S^1)$.

In 2005, H. Pei \cite{Pei} introduced a new semigroup which is a generalization of $T(X)$ as follows. For an equivalence relation $E$ on $X$, the author defined a subsemigroup $T_E(X)$ of $T(X)$ by
\[
T_{E}(X)=\{\alpha\in T(X):\forall x,y\in X,(x,y)\in E\Rightarrow (x\alpha,y\alpha)\in E\}
\] 
and called $T_E(X)$ the \textit{transformation semigroup that preserve equivalence}. Clearly, if $E$ is the universal relation, $X\times X$, then $T_E(X)=T(X)$. Furthermore, in topology, $T_E(X)$ is precisely the semigroup of all continuous self-maps of the topological space $X$ for which all $E$-classes form a basis, denoted by $S(X)$. The regularity of elements and Green's relations for $T_E(X)$ was investigated in \cite{Pei}. Later, in 2008, L. Sun, H. Pei and Z. Cheng \cite{Sun} gave a characterization for the natural partial order $\leq$ on $T_E(X)$.

In 2010, L. Deng, J. Zeng and B. Xu \cite{Deng} defined a subsemigroup $T_{E^*}(X)$ of $T_E(X)$ by adding the reverse direction of the equivalence preserving, namely,
\[
T_{E^*}(X)=\{\alpha\in T(X):\forall x,y\in X,(x,y)\in E\Leftrightarrow (x\alpha,y\alpha)\in E\}.
\]
Similar to $T_E(X)$, if $E$ is the universal relation on $X$, then $T_{E^*}(X)=T(X)$ which implies that this semigroup is also a generalization of $T(X)$. Moreover, $T_{E^*}(X)$ is a semigroup of continuous self-maps of the topological space $X$ for which all $E$-classes form a basis and is referred to as a semigroup of continuous functions (see \cite{Magill} for details). In \cite{Deng}, the authors called $T_{E^*}(X)$ the \textit{transformation semigroup that preserve double direction equivalence}, and then they studied the regularity of its elements and Green's relations. Furthermore, a characterization of the natural partial order $\leq$ on $T_{E^*}(X)$ was investigated by L. Sun and J. Sun \cite{Sun2} in 2013.

Since, in general, the semigroup $T_{E^*}(X)$ is not regular, the aim of this paper is to study the regular part of $T_{E^*}(X)$, denoted by $\reg(T)$. We show in Section \ref{sec: Largest regular subsemigroup and its Green's relations} that $\reg(T)$ is the largest regular subsemigroup of $T_{E^*}(X)$. As we will discuss later on, it is worth studying the algebraic properties of such a regular part in its own right since it is almost never isomorphic to $T(Z)$ for any set $Z$.

In section \ref{sec: Largest regular subsemigroup and its Green's relations}, we characterize Green's relations on $\reg(T)$ by applying Hall's Theorem on the results from \cite{Deng}. Then, in Section \ref{sec: Ideals}, we deal with ideals of $\reg(T)$ and obtain the smallest ideal which becomes the kernel of it. Moreover, we will see that this kernel is a right group which can be written as a union of symmetric groups. Finally, we focus on the embeddability of any right group into the kernel of $\reg(T)$.

Now, we first state some definitions, notations and results which will be used later in our paper. For all undefined notions, the reader is referred to \cite{Clifford1,Clifford,Howie}.

Let $X$ be a nonempty set and $E$ an equivalence relation on $X$. The family of all equivalence classes, denoted $X/E$, is a partition of $X$, i.e., they are mutually disjoint and their union is $X$. For each $\alpha\in T(X)$ and $A\subseteq X$, let $A\alpha=\{a\alpha:a\in A\}$. Evidently, by this notation, we have $X\alpha$ means the range of $\alpha$. The partition of $X$ induced by $\alpha$, denoted $\pi(\alpha)$, is the family of all inverse image of elements in the range of $\alpha$, that is,
\[
\pi(\alpha)=\{ x\alpha^{-1} : x\in X\alpha\}.
\]
It is easy to see that $\pi(\alpha)=X/\ker(\alpha)$ where $\ker(\alpha)=\{(x,y)\in X\times X : x\alpha=y\alpha\}$ is an equivalence relation on $X$. In addition, due to \cite{Deng}, define a mapping $\alpha_*$ from $\pi(\alpha)$ onto $X\alpha$ by 
\[
(x\alpha^{-1})\alpha_*=x\ \text{for each}\ x\in X\alpha.
\]
For each $\alpha\in T(X)$ and $A\subseteq X$, the restriction of $\alpha$ to $A$ is denoted by $\alpha|_A$. 

Let $E$ be an equivalence relation on a set $X$ and let $Y,Z$ be subsets of $X$ and $\varphi$ a mapping from $Y$ into $Z$. If for any $x,y\in Y$, $(x,y)\in E$ implies $(x\varphi,y\varphi)\in E$, then we say that $\varphi$ is {\it $E$-preserving}. If $(x,y)\in E$ if and only if $(x\varphi,y\varphi)\in E$, then $\varphi$ is said to be {\it $E^*$-preserving}.

\section{Largest regular subsemigroups and Green's relations}\label{sec: Largest regular subsemigroup and its Green's relations}

The purpose of this section is to find the largest regular subsemigroup of $T_{E^*}(X)$ and to characterize its Green's relations. The following crucial theorem has been proven in \cite{Deng}.

\begin{thm}\cite[Theorem 3.1]{Deng}
	Let $\alpha\in T_{E^*}(X)$. Then $\alpha$ is regular if and only if $A\cap X\alpha\neq\emptyset$ for any $A\in X/E$.
\end{thm}

For convenience, we denote the regular part of $T_{E^*}(X)$ by $\reg(T)$, that is,
\[
\reg(T)=\{\alpha\in T_{E^*}(X):A\cap X\alpha\neq\emptyset\ \text{for any}\ A\in X/E\}.
\]
We note that $\reg(T)$ contains the identity map $1_X$.

\begin{thm}
	$\reg(T)$ forms a regular subsemigroup of $T_{E^*}(X)$. Consequently, $\reg(T)$ is the largest regular subsemigroup of $T_{E^*}(X)$
\end{thm}
\begin{proof}
	Let $\alpha,\beta\in\reg(T)$. We assert that $\alpha\beta\in\reg(T)$. Let $A\in X/E$. Then $A\cap X\beta\neq\emptyset$. There is $a\in A\cap X\beta$ which implies that $a=x\beta$ for some $x\in X$. Moreover, we have $x\in B$ for some $E$-class $B$ and hence there exists $b\in B\cap X\alpha$ since $B\cap X\alpha\neq\emptyset$. Thus $b=z\alpha$ for some $z\in X$. We obtain $(z\alpha,x)=(b,x)\in E$ from which it follows that $(z\alpha\beta,a)=(z\alpha\beta,x\beta)\in E$ and so $z\alpha\beta\in A\cap X\alpha\beta\neq\emptyset$. Hence $\alpha\beta\in\reg(T)$. Therefore, $\reg(T)$ is a subsemigroup of $T_{E^*}(X)$.

	If $\alpha\in\reg(T)$, then $\alpha\beta\alpha=\alpha$ for some $\beta\in T_{E^*}(X)$, so $(\alpha\beta)^2=\alpha\beta$ and hence, by \cite[Theorem 3.3]{Deng}, $\beta\in\reg(T)$. Thus $\reg(T)$ is a regular subsemigroup of $T_{E^*}(X)$. 
\end{proof}

To consider Green's relations on $\reg(T)$, we adopt the following notations introduced by \cite{Howie}. If $U$ is a subsemigroup of a semigroup $S$ and $a,b\in U$, then $(a,b)\in\mathcal{L}^U$ means that there exist $u,v\in U^1$ such that $ua=b$, $vb=a$ while $(a,b)\in\mathcal{L}^S$ means that there exist $s,t\in S^1$ such that $sa=b$, $tb=a$. Obviously,
\[
\mathcal{L}^U\subseteq\mathcal{L}^S\cap(U\times U)
\]
Similarly, we also use the following notations:
\[
\mathcal{R}^U\subseteq\mathcal{R}^S\cap(U\times U),\quad\mathcal{H}^U\subseteq\mathcal{H}^S\cap(U\times U),
\]
\[
\mathcal{D}^U\subseteq\mathcal{D}^S\cap(U\times U),\quad\mathcal{J}^U\subseteq\mathcal{J}^S\cap(U\times U).
\]
Due to the results given by T. E. Hall, it is well-known that if $U$ is a regular subsemigroup of $S$, we have 
\[
\mathcal{L}^U=\mathcal{L}^S\cap(U\times U),\quad\mathcal{R}^U=\mathcal{R}^S\cap(U\times U),\quad\mathcal{H}^U=\mathcal{H}^S\cap(U\times U)
\]
(see Proposition 2.4.2 of \cite{Howie} for details).

For convenience, we write $\mathcal{L}^T=\mathcal{L}^{T_{E^*}(X)}$ and $\mathcal{L}^R=\mathcal{L}^{\reg(T)}$. The same notation applies to $\mathcal{R}^T,\mathcal{R}^R,\mathcal{H}^T,\mathcal{H}^R,\mathcal{D}^T,\mathcal{D}^R,\mathcal{J}^T$ and $\mathcal{J}^R$.

Let $\alpha\in T_{E^*}(X)$, as in \cite{Deng}, the authors defined $Z(\alpha)=\{A\in X/E : A\cap X\alpha=\emptyset\}$. Clearly, if $\alpha\in\reg(T)$, then $Z(\alpha)=\emptyset$.

Refer to the characterizations of Green's relations on $T_{E^*}(X)$ in \cite{Deng}, the authors obtained the following results. For each $\alpha,\beta\in T_{E^*}(X)$,
\begin{enumerate}[(1)]
	\item $(\alpha,\beta)\in\mathcal{L}^T$ if and only if $X\alpha=X\beta$;
	\item $(\alpha,\beta)\in\mathcal{R}^T$ if and only if $\pi(\alpha)=\pi(\beta)$ and $|Z(\alpha)|=|Z(\beta)|$;
	\item $(\alpha,\beta)\in\mathcal{H}^T$ if and only if $\pi(\alpha)=\pi(\beta)\ \text{and}\ X\alpha=X\beta$;
	\item $(\alpha,\beta)\in\mathcal{D}^T$ if and only if $|Z(\alpha)|=|Z(\beta)|$ and there exists $\delta\in T_{E^*}(X)$ such that $\delta|_{X\alpha} : X\alpha \rightarrow X\beta$ is a bijection;
	\item $(\alpha,\beta)\in\mathcal{J}^T$ if and only if $|X\alpha|=|X\beta|$ and there exist $\rho, \tau\in T_{E^*}(X)$ such that for any $A\in X/E$, $A\alpha\subseteq B\beta\rho$ and $A\beta\subseteq C\alpha\tau$ for some $B,C\in X/E$.
\end{enumerate}

Since $\reg(T)$ is a regular subsemigroup of $T_{E^*}(X)$, we can apply the above results with Hall's Theorem to obtain characterizations of Green's relations on $\reg(T)$. We are now in position to state our main results of this section.

\begin{thm}\label{thm: LRH}
	Let $\alpha,\beta\in\reg(T)$. Then
	\begin{enumerate}[(1)]
		\item $(\alpha,\beta)\in\mathcal{L}^R$ if and only if $X\alpha=X\beta$;
		\item $(\alpha,\beta)\in\mathcal{R}^R$ if and only if $\pi(\alpha)=\pi(\beta)$;
		\item $(\alpha,\beta)\in\mathcal{H}^R$ if and only if $\pi(\alpha)=\pi(\beta)\ \text{and}\ X\alpha=X\beta$.
	\end{enumerate}
\end{thm}
\begin{proof}
	Since $\alpha,\beta\in\reg(T)$, we obtain $|Z(\alpha)|=|Z(\beta)|=0$. It is straightforward to prove (1), (2) and (3) by applying Hall's Theorem.
\end{proof}

The proofs of the below theorems are modified from those in \cite{Deng}.

\begin{thm}\label{thm: D relation}
	Let $\alpha,\beta\in\reg(T)$. Then the following statements are equivalent.
	\begin{enumerate}[(1)]
		\item $(\alpha,\beta)\in\mathcal{D}^R$.
		\item There exists $\delta\in\reg(T)$ such that $\delta|_{X\alpha} : X\alpha \rightarrow X\beta$ is a bijection.
		\item There exists $\delta\in T_{E^*}(X)$ such that $\delta|_{X\alpha} : X\alpha \rightarrow X\beta$ is a bijection.
	\end{enumerate}
\end{thm}
\begin{proof}
	(1)$\Rightarrow$(2). Suppose that $(\alpha,\beta)\in\mathscr{D}^R$. Then $(\alpha,\gamma)\in\mathscr{L}^R$ and $(\gamma,\beta)\in\mathscr{R}^R$ for some $\gamma\in\reg(T)$. By Theorem \ref{thm: LRH} (1) and (2), $X\alpha=X\gamma$ and $\pi(\gamma)=\pi(\beta)$. For each $A\in X/E$, we have $A\cap X\gamma\neq\emptyset$ since $\gamma\in\reg(T)$. Hence, for each $A\in X/E$, we can choose and fix $b_A\in(A\cap X\gamma)\gamma^{-1}\beta$ and define $\delta:X\to X$ by 
	\[
	x\delta =\begin{cases} ~~x\gamma^{-1}\beta_* &  \begin{text}{if $x\in A\cap X\gamma$}\end{text}\\ 
	~~b_A &  \begin{text}{if $x\in A\setminus X\gamma$.}\end{text}\end{cases}
	\]
	The proof that $\delta\in T_{E^*}(X)$ is routine. Next, we will show that $\delta\in\reg(T)$. Let $C\in X/E$. Then there is $a\beta\in X\beta\cap C\neq\emptyset$ since $\beta\in\reg(T)$. Let $a\gamma=c$ for some $c\in X$. We see that $c\in X\gamma$ and then $c\delta=c\gamma^{-1}\beta_*$. Furthermore, we have $a\in c\gamma^{-1}\in\pi(\gamma)=\pi(\beta)$. Hence $c\gamma^{-1}\beta_*=a\beta\in C$ and so $c\delta=a\beta\in C\cap X\delta\neq\emptyset$. Therefore, $\delta\in\reg(T)$. Finally, by the same proof as given in Theorem 2.4 of \cite{Deng}, it is concluded that $\delta|_{X\alpha}$ is bijective.
	
	The implication (2)$\Rightarrow$(3) is clear.
	
	(3)$\Rightarrow$(1). Suppose that there exists $\delta\in T_{E^*}(X)$ such that $\delta|_{X\alpha} : X\alpha \rightarrow X\beta$ is a bijection. Define $\gamma : X \rightarrow X$ by $x\gamma =x\beta\delta^{-1}\in X\alpha$. Then $\gamma\in T_{E^*}(X)$ since $\beta$ and $\delta$ are $E^*$-preserving. We assert that $\gamma\in\reg(T)$. Indeed, let $A\in X/E$. Then $X\gamma=X\beta\delta^{-1}=X\alpha$ and $\delta$ is a bijection. Moreover, since $\alpha\in\reg(T)$, we obtain $A\cap X\alpha$ is nonempty. Thus $A\cap X\gamma=A\cap X\alpha$ is also nonempty from which it follows that $\gamma\in\reg(T)$. Again by using the same proof as given in Theorem 2.4 of \cite{Deng}, we obtain $(\alpha,\gamma)\in\mathscr{L}^R$ and $(\gamma,\beta)\in\mathscr{R}^R$. Consequently, $(\alpha,\beta)\in\mathscr{D}^R$.
\end{proof}

\begin{thm}
	Let $\alpha,\beta\in\reg(T)$. Then the following statements are equivalent.
	\begin{enumerate}[(1)]
		\item $(\alpha,\beta)\in\mathcal{J}^R$.
		\item $|X\alpha|=|X\beta|$ and there exist $\rho, \tau\in\reg(T)$ such that for any $A\in X/E$, $A\alpha\subseteq B\beta\rho$ and $A\beta\subseteq C\alpha\tau$ for some $B,C\in X/E$.
	\end{enumerate}
\end{thm}
\begin{proof}
	The proof is an appropriate modification of the proof of Theorem 2.5 of \cite{Deng}. We will show the proof of (2)$\Rightarrow$(1) only. Define a function $\theta:X\to X$ by $x\theta\in x\alpha\rho^{-1}\beta^{-1}$. By the same proof as given in Theorem 2.5 of \cite{Deng}, we have $\theta\in T_{E^*}(X)$ and $\theta\beta\rho=\alpha$. It remains to show that $\theta\in\reg(T)$. Let $A\in X/E$. Then $A\beta\rho\subseteq B$ for some $E$-class $B$. Since $\alpha\in\reg(T)$, we obtain $B\cap X\alpha$ is nonempty. Then there exists $b\in B\cap X\alpha$ and hence $b=c\alpha$ for some $c\in X$. Assume that $c\in C$ for some $E$-class $C$. By assumption, $C\alpha\subseteq D\beta\rho$ for some $D\in X/E$ which implies that $c\alpha=d\beta\rho$ for some $d\in D$. Thus $d\beta\rho=c\alpha\in B$ from which it follows that $d\in A$ since $A\beta\rho\subseteq B$ and $\beta\rho$ is $E^*$-preserving. We obtain $A=D$ and so $C\alpha\subseteq A\beta\rho$. By the definition of $\theta$, we have $c\theta\in c\alpha\rho^{-1}\beta^{-1}$. Hence $c\theta\beta\rho=c\alpha\in C\alpha\subseteq A\beta\rho$ from which it follows that $c\theta\in A$. Therefore, $c\theta\in A\cap X\theta\neq\emptyset$ and so $\theta\in\reg(T)$, as required. Similarly, we can construct a function $\mu\in\reg(T)$ such that $\mu\alpha\tau=\beta$. Therefore, $(\alpha,\beta)\in\mathscr{J}^R$.
\end{proof}

\section{Ideals}\label{sec: Ideals}

In this section, we focus our attention on (two-sided) ideals of $\reg(T)$ and find the smallest ideal of it. Let $X/E=\{A_i:i\in I\}$. We note that for each $\beta\in\reg(T)$ and $i\in I$, we obtain $A_i\beta\subseteq A_j$ for some $j\in I$. Define a function $\tau:I\to I$ by $i\mapsto j$. It is straightforward to verify that $\tau$ is a permutation on $I$ since $\beta$ is $E^*$-preserving and regular.

To describe ideals of $\reg(T)$, the following lemma is needed.

\begin{lem}\label{lem: pre ideal}
	Let $X/E=\{A_i:i\in I\}$ and $\alpha,\beta\in\reg(T)$. Then the following statements are equivalent.
	\begin{enumerate}
		\item There exists a permutation $\rho$ on the index set $I$ such that $|A_i\alpha|\leq|A_{i\rho}\beta|$ for all $i\in I$.
		\item $\alpha=\lambda\beta\mu$ for some $\lambda,\mu\in\reg(T)$.
	\end{enumerate}
\end{lem}
\begin{proof}
	(1) $\Rightarrow$ (2). Assume that there exists a permutation $\rho$ on the index set $I$ such that $|A_i\alpha|\leq|A_{i\rho}\beta|$ for all $i\in I$. Then, for each $i\in I$, there is an injection $\sigma_i:A_i\alpha\to A_{i\rho}\beta$. Define a function $\sigma:X\alpha\to X\beta$ by $\sigma=\bigcup_{i\in I}\sigma_i$. This function is well-defined and one-to-one since $\alpha$ and $\beta$ are $E^*$-preserving.
	
	As mentioned above, there is a permutation $\tau$ on $I$ such that $A_{i\rho}\beta\subseteq A_{i\tau}$ for all $i\in I$. For each $i\in I$, choose an element $y_0\in A_i\alpha$ and define a surjective function $\mu_i:A_{i\tau}\to A_i\alpha$ by
	\[
	x\mu_i =\begin{cases} ~~x\sigma_i^{-1} &  \begin{text}{if $x\in A_{i\rho}\beta$}\end{text}\\ 
	~~y_0 &  \begin{text}{if $x\in A_{i\tau}\setminus A_{i\rho}\beta$.}\end{text}\end{cases}
	\]
	Since $\beta$ is regular and $E^*$-preserving, we obtain $\bigcup_{i\in I}A_{i\tau}=X$ and $A_{i\tau}\cap A_{j\tau}=\emptyset$ for each distinct indices $i$ and $j$ in $I$. Now, we define a function $\mu:X\to X$ by $\mu=\bigcup_{i\in I}\mu_i$. We note that $X\mu=\bigcup_{i\in I} A_i\alpha=X\alpha$.

	Let $(x,y)\in E$. Then $x,y\in A_{i\tau}$ for some $i\in I$. We have $x\mu=x\mu_i\in A_i\alpha$ and $y\mu=y\mu_i\in A_i\alpha$ from which it follows that $(x\mu,y\mu)\in E$. On the other hand, let $(x\mu,y\mu)\in E$. If $x\mu\in A_i\alpha$ and $y\mu\in A_j\alpha$ for some $i,j\in I$, then $i=j$ since $\alpha$ is $E^*$-preserving. Hence $x,y\in A_{i\tau}$ and so $(x,y)\in E$. Therefore, $\mu\in T_{E^*}(X)$. Moreover, since $\alpha$ is regular, we have $A_i\cap X\mu=A_i\cap X\alpha\neq\emptyset$ for each $i\in I$ which implies that $\mu$ is regular.

	For each $x\in\im\sigma$, choose an element $x'\in x\beta^{-1}$ (possible since $\im\sigma\subseteq X\beta$). Define functions $\gamma:\im\sigma\to X$ and $\lambda:X\to X$ by $x\gamma=x'$ and $\lambda=\alpha\sigma\gamma$, respectively. We claim that $A_i\lambda\subseteq A_{i\rho}$ for all $i\in I$. For, let $i\in I$ and $x\in A_i$. Then $x\alpha\in A_i\alpha$ and hence $x\alpha\sigma=x\alpha\sigma_i\in A_{i\rho}\beta$. So $x\lambda=x\alpha\sigma\gamma\in(x\alpha\sigma)\beta^{-1}\subseteq(A_{i\rho}\beta)\beta^{-1}=A_{i\rho}$. It is concluded that $\lambda$ is $E^*$-preserving and regular.

	Finally, we show that $\alpha=\lambda\beta\mu$. Let $x\in X$. Then $x\in A_i$ for some $i\in I$ and $x\alpha\in A_i\alpha$. We have $x\alpha\sigma=x\alpha\sigma_i=y$ for some $y\in A_{i\rho}\beta$. By the definition of $\mu$, we can see that $y\mu=y\sigma_i^{-1}$. Hence $x\lambda=x\alpha\sigma\gamma=y\gamma=y'\in y\beta^{-1}$ which implies that
	\[
	x\lambda\beta\mu=y'\beta\mu=y\mu=y\sigma_i^{-1}=(x\alpha\sigma_i)\sigma_i^{-1}=x\alpha.
	\]
	
	(2) $\Rightarrow$ (1). Let $\alpha=\lambda\beta\mu$ for some $\lambda,\mu\in\reg(T)$. Then there exists a permutation $\rho$ on $I$ such that $A_i\lambda\subseteq A_{i\rho}$ for all $i\in I$. Therefore,
	\[
	|A_i\alpha|=|A_i\lambda\beta\mu|\leq|A_{i\rho}\beta\mu|\leq|A_{i\rho}\beta|
	\]
	for all $i\in I$, as required.
\end{proof}

For each cardinal number $r$, we let $r'$ denote the successor of $r$. Let $X/E=\{A_i:i\in I\}$ and let $\mathbf{r}=\{r_i\}_{i\in I}$ be a set of cardinal numbers such that $2\leq r_i\leq|A_i|'$ for all $i\in I$. Denote the symmetric group on the index set $I$ by $S_I$, we then define
\[
Q(\mathbf{r})=\{\alpha\in\reg(T):\exists\sigma\in S_I,\forall i\in I[|A_i\alpha|<r_{i\sigma}]\}.
\]
Moreover, we assert that $Q(\mathbf{r})$ is an ideal of $\reg(T)$. For, let $\alpha\in Q(\mathbf{r})$ and $\beta\in\reg(T)$. Then there exists a permutation $\sigma$ on $I$ such that $|A_i\alpha|<r_{i\sigma}$ for all $i\in I$. As we mentioned before, there is a permutation $\tau$ on $I$ such that $A_i\beta\subseteq A_{i\tau}$ for all $i\in I$. We obtain $|A_i\alpha\beta|\leq|A_i\alpha|<r_{i\sigma}$ and $|A_i\beta\alpha|\leq|A_{i\tau}\alpha|<r_{i\tau\sigma}$ for all $i\in I$. Then $\alpha\beta,\beta\alpha\in Q(\mathbf{r})$ which implies that $Q(\mathbf{r})$ is an ideal of $\reg(T)$.

In the next theorem, we shall use the symbols $Q(\mathbf{r}^\alpha)$ and $\{r^\alpha_i\}_{i\in I}$. The superscript $\alpha$ in $\mathbf{r}^\alpha$ and $r^\alpha_i$ does not mean that $\mathbf{r}$ and $r_i$ are raised to the power $\alpha$, but is simply an index to distinguish among variables.

\begin{thm}\label{thm: ideal}
	Let $X/E=\{A_i:i\in I\}$. The ideals of $\reg(T)$ are precisely the set
	\[
	\bigcup_{\alpha\in J}Q(\mathbf{r}^\alpha)
	\]
	for some index set $J$ and for each $\alpha\in J$, $\mathbf{r}^\alpha=\{r^\alpha_i\}_{i\in I}$ is a set of cardinal numbers defined as above. Furthermore, $\bigcup_{\alpha\in J}Q(\mathbf{r}^\alpha)$ is principal if $J=\{\alpha\}$ is a singleton and $\mathbf{r}^\alpha=\{r^\alpha_i\}_{i\in I}$ such that $r^\alpha_i=s'_i$ where $1\leq s_i\leq|A_i|$ for all $i\in I$.
\end{thm}
\begin{proof}
	Obviously, $\bigcup_{\alpha\in J}Q(\mathbf{r}^\alpha)$ is a union of ideals of $\reg(T)$ which is also an ideal. Conversely, suppose that $J$ is an ideal of $\reg(T)$. For each $\alpha\in J$, let $\mathbf{r^\alpha}=\{r^\alpha_i\}_{i\in I}$ be such that $r^\alpha_i$ is the least cardinal greater than $|A_i\alpha|$ for all $i\in I$. We can see that for each $\alpha\in J$, $\alpha\in Q(\mathbf{r}^\alpha)$ (respect to the identity function as a permutation on $I$). Then $J\subseteq\bigcup_{\alpha\in J}Q(\mathbf{r}^\alpha)$. Let $\beta\in\bigcup_{\alpha\in J}Q(\mathbf{r}^\alpha)$. Then $\beta\in Q(\mathbf{r}^\alpha)$ for some $\alpha\in J$. There exists permutation $\tau$ on $I$ such that $|A_i\beta|=s_i<r^\alpha_{i\tau}$ for all $i\in I$. Let $\rho$ be the inverse function of $\tau$. Suppose to the contrary that $|A_i\alpha|<s_{i\rho}$ for some $i\in I$. We obtain
	\[
	|A_i\alpha|<s_{i\rho}<r^\alpha_{i\rho\tau}=r^\alpha_i
	\]
	which contradicts to the choice of $r^\alpha_i$. Hence $|A_i\alpha|\geq s_{i\rho}=|A_{i\rho}\beta|$ for all $i\in I$. By Lemma \ref{lem: pre ideal}, there are $\lambda,\mu\in\reg(T)$ such that $\beta=\lambda\alpha\mu$. Since $\alpha\in J$ and $J$ is an ideal, we get $\beta=\lambda\alpha\mu\in J$ and so $\bigcup_{\alpha\in J}Q(\mathbf{r}^\alpha)\subseteq J$.
	
	Next, we prove the second part of the theorem. Assume that $J=\{\alpha\}$ is a singleton and $\mathbf{r}^\alpha=\{r^\alpha_i\}_{i\in I}$ such that $r^\alpha_i=s'_i$ where $1\leq s_i\leq|A_i|$. Then $\bigcup_{\alpha\in J}Q(\mathbf{r}^\alpha)=Q(\mathbf{r}^\alpha)$. For each $i\in I$, since $1\leq s_i\leq|A_i|$, there is a function $\gamma_i:A_i\to A_i$ such that $|A_i\gamma_i|=s_i$. Define $\gamma:X\to X$ by $\gamma=\bigcup_{i\in I}\gamma_i$. It is clear that $\gamma$ is $E^*$-preserving and regular. Moreover, we obtain $\gamma\in Q(\mathbf{r}^\alpha)$ since $|A_i\gamma|=|A_i\gamma_i|=s_i<s'_i=r^\alpha_i$ for all $i\in I$ (respect to the identity function as a permutation on $I$). We assert that $Q(\mathbf{r}^\alpha)=\reg(T)\gamma\reg(T)$. For, let $\beta\in Q(\mathbf{r}^\alpha)$. Then there exists a permutation $\sigma$ on $I$ such that $|A_i\beta|<r^\alpha_{i\sigma}=s'_{i\sigma}$ for all $i\in I$. We obtain $|A_i\beta|\leq s_{i\sigma}=|A_{i\sigma}\gamma|$ for all $i\in I$. Hence $\beta=\lambda\gamma\mu$ for some $\lambda,\mu\in\reg(T)$ by Lemma \ref{lem: pre ideal}. So $\beta\in\reg(T)\gamma\reg(T)$ which implies that $Q(\mathbf{r}^\alpha)\subseteq\reg(T)\gamma\reg(T)$. The other containment is clear since $\gamma\in Q(\mathbf{r}^\alpha)$ and $Q(\mathbf{r}^\alpha)$ is an ideal. It follows that $\bigcup_{\alpha\in J}Q(\mathbf{r}^\alpha)=Q(\mathbf{r}^\alpha)$ is principal.
\end{proof}

The following theorem demonstrates a characterization of principal ideals of $\reg(T)$ when the set of all equivalence classes, $X/E$, is finite.

\begin{thm}
	Let $X/E=\{A_i:i\in I\}$ such that $I$ is finite. If $\bigcup_{\alpha\in J}Q(\mathbf{r}^\alpha)$ is a principal ideal, then $J=\{\alpha\}$ is a singleton and $\mathbf{r}^\alpha=\{r^\alpha_i\}_{i\in I}$ such that $r^\alpha_i=s'_i$ where $1\leq s_i\leq|A_i|$ for all $i\in I$.
\end{thm}
\begin{proof}	
	Suppose that $\bigcup_{\alpha\in J}Q(\mathbf{r}^\alpha)=\reg(T)\gamma\reg(T)$ for some $\gamma\in\bigcup_{\alpha\in J}Q(\mathbf{r}^\alpha)$. Then $\gamma\in Q(\mathbf{r}^\alpha)$ for some $\alpha\in J$ from which it follows that $\reg(T)\gamma\reg(T)\subseteq Q(\mathbf{r}^\alpha)$ since $Q(\mathbf{r}^\alpha)$ is an ideal. Hence
	\[
	\bigcup_{\alpha\in J}Q(\mathbf{r}^\alpha)=\reg(T)\gamma\reg(T)\subseteq Q(\mathbf{r}^\alpha)\subseteq\bigcup_{\alpha\in J}Q(\mathbf{r}^\alpha)
	\]
	which implies that $\bigcup_{\alpha\in J}Q(\mathbf{r}^\alpha)=\reg(T)\gamma\reg(T)=Q(\mathbf{r}^\alpha)$ and so $J=\{\alpha\}$ is a singleton. Let $|A_i\gamma|=t_i$ for all $i\in I$. Since $\gamma\in Q(\mathbf{r}^\alpha)$, there is a permutation $\sigma$ on $I$ such that $t_i<r^\alpha_{i\sigma}$ for all $i\in I$. Let $i\in I$. We can see that $|A_i\gamma|=t_i<r^\alpha_{i\sigma}\leq|A_{i\sigma}|'$ and hence $t_i\leq|A_{i\sigma}|$. Then there is a subset $B_{i\sigma}$ of $A_{i\sigma}$ such that $|B_{i\sigma}|=t_i$. Let $\phi_{i\sigma}$ be a surjection from $A_{i\sigma}$ to $B_{i\sigma}$ for all $i\in I$. Suppose to a contrary that there exists $i_0\in I$ and a cardinal number $r$ such that 
	\[
	|A_{i_0}\gamma|=t_{i_0}<r<r^\alpha_{i_0\sigma}\leq|A_{i_0\sigma}|'.
	\]
	Then $r\leq|A_{i_0\sigma}|$. We can choose a subset $B_{i_0\sigma}$ of $A_{i_0\sigma}$, where $|B_{i_0\sigma}|=r$. Let $\phi_{i_0\sigma}$ be a surjection from $A_{i_0\sigma}$ to $B_{i_0\sigma}$. Define 
	\[
	x\beta =\begin{cases} ~~x\phi_{i\sigma} &  \begin{text}{if $x\in A_{i\sigma}$ and $i\neq i_0$}\end{text}\\ 
	~~x\phi_{i_0\sigma} &  \begin{text}{if $x\in A_{i_0\sigma}$.}\end{text}\end{cases}
	\]
	Clearly, for each $i\in I$, $A_{i\sigma}\beta\subseteq A_{i\sigma}$ which implies that $\beta$ is $E^*$-preserving and regular. We note that for each $i\neq i_0$, $|A_{i\sigma}\beta|=|B_{i\sigma}|=t_i=|A_i\gamma|$. Moreover, we can see that $|A_{i_0\sigma}\beta|=|B_{i_0\sigma}|=r<r^\alpha_{i_0\sigma}$ and $|A_{i\sigma}\beta|=|B_{i\sigma}|=t_i<r^\alpha_{i\sigma}$ for all $i\neq i_0$. Thus $\beta\in Q(\mathbf{r}^\alpha)=\reg(T)\gamma\reg(T)$ and hence $\beta=\lambda\gamma\mu$ for some $\lambda,\mu\in\reg(T)$. By Lemma \ref{lem: pre ideal}, there exists a permutation $\rho$ on $I$ such that $|A_{i\sigma}\beta|\leq|A_{i\sigma\rho}\gamma|$ for all $i\in I$. Since $I$ is finite, there is the smallest natural number $n$ such that $i_0(\sigma\rho)^n=i_0$. Hence
	\[
	|A_{i_0\sigma\rho}\gamma|=|A_{i_0\sigma\rho\sigma}\beta|\leq|A_{i_0(\sigma\rho)^2}\gamma|,
	\]
	\[
	|A_{i_0(\sigma\rho)^2}\gamma|=|A_{i_0(\sigma\rho)^2\sigma}\beta|\leq|A_{i_0(\sigma\rho)^3}\gamma|,
	\]
	\[
	\vdots
	\]
	\[
	|A_{i_0(\sigma\rho)^{n-1}}\gamma|=|A_{i_0(\sigma\rho)^{n-1}\sigma}\beta|\leq|A_{i_0(\sigma\rho)^n}\gamma|.
	\]
	So
	\[
	r=|A_{i_0\sigma}\beta|\leq|A_{i_0\sigma\rho}\gamma|\leq|A_{i_0(\sigma\rho)^n}\gamma|=|A_{i_0}\gamma|=t_{i_0}<r
	\]
	which is a contradiction. We conclude that $r^\alpha_{i\sigma}=t'_i\leq|A_{i\sigma}|'$ for all $i\in I$. Therefore, $r^\alpha_i=s'_i$ where $1\leq s_i\leq|A_i|$ for all $i\in I$.
\end{proof}

In the previous theorem, the proof of the statement "If $\bigcup_{\alpha\in J}Q(\mathbf{r}^\alpha)$ is a principal ideal, then $J=\{\alpha\}$ is a singleton." does not depend on the finite property of $I$. So, to sum it up, $\bigcup_{\alpha\in J}Q(\mathbf{r}^\alpha)$ is a principal ideal if and only if $J=\{\alpha\}$ is a singleton.

From now on, we adopt the notation introduced by \cite{Clifford}, namely, if $\alpha\in T(X)$, then we write
\[
\alpha = \left(
\begin{array}{c}
	X_i\\
	a_i\\
\end{array}
\right)
\]
and take as understood that the subscript $i$ belongs to some unmentioned index set $I$, the abbreviation $\{a_i\}$ denotes $\{a_i:i\in I\}$, and that $X\alpha=\{a_i\}$ and $a_i\alpha^{-1}=X_i$.

In the next theorem, and later in this paper, we denote the set of cardinal numbers $\{r_i\}_{i\in I}$ where $r_i=2$ for all $i\in I$ by $\mathbf{2}$. The below theorem shows that $\reg(T)$ has the smallest ideal.

\begin{thm}\label{thm: minimal ideal}
	$Q(\mathbf{2})$ is the smallest ideal of $\reg(T)$. Consequently, it is the (unique) minimal ideal of $\reg(T)$.
\end{thm}
\begin{proof}
	Let $X/E=\{A_i:i\in I\}$. For each $i\in I$, we fix $y_i\in A_i$ and then define a function $\gamma:X\to X$ by 
	\[
	\gamma= \left(
	\begin{array}{c}
	A_i \\
	y_i \\ 
	\end{array}
	\right).
	\]
	It is obvious that $A_i\gamma\subseteq A_i$ and $|A_i\gamma|=1<2$ for all $i\in I$. Hence $\gamma\in Q(\mathbf{2})\neq\emptyset$. Assume that $\bigcup_{\alpha\in J}Q(\mathbf{r}^\alpha)$ defined as same as Theorem \ref{thm: ideal} is an ideal of $\reg(T)$. Let $\beta\in Q(\mathbf{2})$. Then, for each $\alpha\in J$, we obtain $|A_i\beta|<2\leq r^\alpha_i$ for all $i\in I$. Hence $\beta\in Q(\mathbf{r}^\alpha)\subseteq\bigcup_{\alpha\in J}Q(\mathbf{r}^\alpha)$ which implies that $Q(\mathbf{2})$ is the smallest ideal of $\reg(T)$.
\end{proof}

According to \cite{Clifford1}, we call the minimal ideal $Q(\mathbf{2})$ the \textit{kernel} of $\reg(T)$. It is obvious that $Q(\mathbf{2})=\{\alpha\in\reg(T):|A\alpha|=1\ \text{for all}\ A\in X/E\}$ which implies that $\pi(\alpha)=X/E$ for all $\alpha\in Q(\mathbf{2})$. Moreover, for each $\alpha\in Q(\mathbf{2})$, $|X\alpha\cap A|=1$ for all $A\in X/E$, equivalently, $X\alpha$ is a \textit{cross section} of the partition
$X/E$ induced by $E$, i.e., each $E$-class contains exactly one element of $X\alpha$.

Recall from \cite[pp. 37-40]{Clifford1} that a semigroup $S$ is called a \textit{right simple} if it contains no proper right ideal and a semigroup $S$ is called a \textit{right group} if it is right simple and left cancellative. This is equivalent to saying that $S$ is a right group if and only if, for any elements $a$ and $b$ of $S$, there exists one and only one element $x$ of $S$ such that $ax=b$. Additional descriptions are contained in the following lemma which is a combination of statements contained in Exercises 2 and 4 for \S1.11 of \cite{Clifford1}.

\begin{lem}
	Let $S$ be a semigroup. The following statements are equivalent.
	\begin{enumerate}[(1)]
		\item $S$ is a right group.
		\item $S$ is a union of disjoint groups such that the set of identity elements of the groups is a right zero subsemigroup of $S$.
		\item $S$ is regular and left cancellative.
	\end{enumerate}
\end{lem}

\begin{thm}
	$Q(\mathbf{2})$ is a right group.
\end{thm}
\begin{proof}
	Assume that $X/E=\{A_i:i\in I\}$. Let $\alpha\in Q(\mathbf{2})$. We can write
	\[
	\alpha= \left(
	\begin{array}{c}
		A_i \\
		a_i \\ 
	\end{array}
	\right).
	\]
	For each $i\in I$, choose $b_i\in A_i$. Define a function $\beta$ by
	\[
	\beta= \left(
	\begin{array}{c}
		[a_i] \\
		b_i \\ 
	\end{array}
	\right)
	\]
	where $[a_i]$ is an $E$-class containing the element $a_i$. It is easy to see that $\beta\in Q(\mathbf{2})$	and $\alpha=\alpha\beta\alpha$. Therefore, $Q(\mathbf{2})$ is regular.
	
	To prove the left cancellation, suppose that $\alpha,\beta,\gamma\in Q(\mathbf{2})$ such that $\gamma\alpha=\gamma\beta$. Let $x\in X$. Then $x\in A$ for some $A\in X/E$. Choose $y\in X$ such that $y\gamma\in A$ (possible since $A\cap X\gamma$ is nonempty). We obtain 
	\[
	x\alpha=y\gamma\alpha=y\gamma\beta=x\beta
	\]
	and hence $\alpha=\beta$, as required. 
\end{proof}

Applying Theorem \ref{thm: LRH} (3) with Hall's Theorem, we obtain the following result.

\begin{prop}
	For each $\alpha,\beta\in Q(\mathbf{2})$, $(\alpha,\beta)\in\mathcal{H}^{Q(\mathbf{2})}$ if and only if $X\alpha=X\beta$.
\end{prop}

Refer to Exercise 2 for \S2.1 of \cite{Clifford1}, every $\mathcal{H}$-class of a right group $S$ is a subgroup of $S$. As a direct consequence of the above proposition, we get the following corollary immediately.

\begin{cor}\label{thm: subgroup of Q}
	For each $\alpha\in Q(\mathbf{2})$, $H_\alpha=\{\beta\in Q(\mathbf{2}):X\alpha=X\beta\}$ forms a subgroup of $Q(\mathbf{2})$.
\end{cor}

In general, a right group is a union of disjoint groups. For $Q(\mathbf{2})$, we obtain the following result.

\begin{thm}
	$Q(\mathbf{2})$ is a union of symmetric groups.
\end{thm}
\begin{proof}
	By Corollary \ref{thm: subgroup of Q}, we have $Q(\mathbf{2})$ is a union of group $H_\alpha=\{\beta\in Q(\mathbf{2}):X\alpha=X\beta\}$ for all $\alpha\in Q(\mathbf{2})$. It suffices to show that $H_\alpha$ is isomorphic to a symmetric group. Let $\beta\in H_\alpha$. Then $\beta\in Q(\mathbf{2})$ and $X\alpha=X\beta$. Since $\beta$ is $E^*$-preserving, one can easily see that $\beta$ is a one-to-one correspondence between $X/E$ and $X\alpha$. Hence it is straightforward to verify that $H_\alpha$ is isomorphic to the symmetric group on $X\alpha$.
\end{proof}

\section{Isomorphism theorem and embedding problem}

It is well-known that the ideals of $T(X)$ are precisely the set
\[
I_r=\{\alpha\in T(X):|X\alpha|<r\}
\]
where $2\leq r\leq|X|'$ (see pp. 227 of \cite{Clifford} for details). It is clear that $I_2\subseteq I_3\subseteq\cdots$ and hence it forms a chain under inclusion. We conclude that $I_2$, the set of all constant functions from $X$ into $X$, is the (unique) minimal ideal of $T(X)$.

\begin{lem}
	$\alpha\in Q(\mathbf{2})$ is an idempotent if and only if $A\alpha\subseteq A$ for all $A\in X/E$.
\end{lem}
\begin{proof}
	Suppose that $\alpha\in Q(\mathbf{2})$ is an idempotent. Let $A\in X/E$ and $a\in A$. Then $(a\alpha)\alpha=a\alpha^2=a\alpha$. Thus $(a\alpha,a)\in E$ and so $a\alpha\in A$. Consequently, $A\alpha\subseteq A$. Conversely, assume that $A\alpha\subseteq A$ for all $A\in X/E$. Let $x\in X$. Then $x\in A$ for some $A\in X/E$. We obtain $x\alpha\in A\alpha\subseteq A$. Hence $x\alpha^2=x\alpha$ since $|A\alpha|=1$.
\end{proof}

By the above lemma, $\alpha\in Q(\mathbf{2})$ is an idempotent if and only if for each $A\in X/E$, there is $x\in A$ such that $A\alpha=\{x\}$. The next result shows that $\reg(T)$ is almost never isomorphic to $T(Z)$ for any set $Z$.

\begin{thm}
	$\reg(T)\cong T(Z)$ for some set $Z$ if and only if $E=X\times X$.
\end{thm}
\begin{proof}
	($\Leftarrow$). Assume that $E=X\times X$. Then $\reg(T)=T(X)$.
	
	($\Rightarrow$). Suppose that $\reg(T)\cong T(Z)$ for some set $Z$. Then there is an isomorphism $\varphi:T(Z)\to\reg(T)$. Let $I_2$ be the minimal ideal of $T(Z)$ as defined above. For each $\alpha\in I_2$, we can write
	\[
	\alpha= \left(
	\begin{array}{c}
	Z \\
	a \\ 
	\end{array}
	\right)
	\]
	where $a\in Z$ and hence, clearly, it is an idempotent. Thus $I_2\varphi$ is the minimal ideal in $\reg(T)$ which implies by Theorem \ref{thm: minimal ideal} that $I_2\varphi=Q(\mathbf{2})$. Thus every element in $Q(\mathbf{2})$ is also an idempotent. Suppose to a contrary that $E\neq X\times X$. Then $|X/E|>1$. Let $X/E=\{A_i:i\in I\}$ and let $A_{i_1},A_{i_2}$ be distinct $E$-classes in $X/E$. Choose $a\in A_{i_1}$, $b\in A_{i_2}$ and $c_i\in A_i$ for each $i\in I$ such that $i_1\neq i\neq i_2$. Define a function $\beta:X\to X$ By
	\[
	x\beta =\begin{cases} ~~b &  \begin{text}{if $x\in A_{i_1}$}\end{text}\\ 
	~~a &  \begin{text}{if $x\in A_{i_2}$}\end{text}\\
	~~c_i &  \begin{text}{if $x\in A_i$ and $i_1\neq i\neq i_2$.}\end{text}\end{cases}
	\]
	It is easy to see that $\beta\in Q(\mathbf{2})$ but not an idempotent which is a contradiction. Therefore, $E=X\times X$.
\end{proof}

Finally, we deal with the embeddability of a right group $S$ into the kernel of $\reg(T)$. Refer to \cite[Exercise 6 for \S2.6]{Howie} and \cite[Exercise 3 for \S1.11]{Clifford1}, it is straightforward to verify that every right group $S$ is the union of a set $\{Se:e\in E(S)\}$ of isomorphic disjoint groups and a subgroup $Se$ of $S$ contains an idempotent $e$ as the group identity.

The \textit{inner right translation} on a semigroup $S$ corresponding to the element $a$ of $S$ is the mapping $\rho_a:S\to S$ defined by $x\mapsto xa$. One can easily see that, for each $a,b\in S$, $\rho_a\rho_b=\rho_{ab}$ and so the set $\{\rho_a:a\in S\}$ is a subsemigroup of $T(S)$ under the composition. The mapping $a\mapsto\rho_a$ is called the \textit{regular representation} of $S$.

Given a right group $S$, it is a natural question whether there exists the subsemigroup $T$ of $T(S)$ such that $S$ can be embedded into $T$ and $T$ is also a right group. To answer this question, we are in position to consider the embedding problem for the right group.

\begin{thm}
	Any right group $S$ can be embedded into the subsemigroup $Q(\mathbf{2})$ of $\reg(T_{E^*}(S))$ for some equivalence relation $E$ on $S$.
\end{thm}
\begin{proof}
	Let $E(S)=\{e_i:i\in I\}$ be the set of all idempotents in $S$. Then $S$ the union of a set $\{Se_i:i\in I\}$ of isomorphic disjoint groups. We note that for each $i\in I$, the element $e_i$ is the group identity of $Se_i$. For each $s\in S$, define a subset $[s]$ of $S$ by
	\[
	[s]=\{se_i:i\in I\}.
	\]
	We claim that the class $\{[s]:s\in S\}$ forms a partition of $S$. Clearly, $\bigcup_{s\in S}[s]=S$. If $x\in[s]\cap[t]$, then $x=se_i=te_j$ for some $i,j\in I$ which implies that $x\in Se_i\cap Se_j$. Hence $i=j$ and so $se_i=te_i$. Let $y\in[s]$. Then $y=se_k$ for some $k\in I$. By \cite[Exercise 6 for \S2.6]{Howie}, we have $E(S)$ is a right zero subsemigroup of $S$. It follows that
	\[
	y=se_k=se_ie_k=te_ie_k=te_k\in[t]
	\] 
	and thus $[s]\subseteq[t]$. Similarly, we can show that $[t]\subseteq[s]$ and so $[s]=[t]$. Define an equivalence relation $E$ on $S$ by $S/E=\{[s]:s\in S\}$. 
	
	Next, we shall construct a semigroup monomorphism from the right group $S$ into the subsemigroup $Q(\mathbf{2})$ of $\reg(T)=\reg(T_{E^*}(S))$. For each $s\in S$, let $\rho_s:S\to S$ be the inner right translation on $S$ corresponding to $s$, that is, $x\mapsto xs$. We first show that $\rho_s\in Q(\mathbf{2})$.  Assume that $s\in Se_k$ for some $k\in I$.
	
	Let $x,y\in[t]$. Then $x=te_i$ and $y=te_j$ for some $i,j\in I$. We obtain $x\rho_s=xs=te_is=ts$. Similarly, we can show that $y\rho_s=ts=x\rho_s$. Therefore, $\rho_s$ is $E$-preserving and $|[t]\rho_s|=1$ for all $t\in S$. 
	
	Conversely, let $x\rho_s,y\rho_s\in[t]$. Then $xs=te_i$ and $ys=te_j$ for some $i,j\in I$. We can see that
	\[
	xs=xse_k=te_ie_k=te_k=te_je_k=yse_k=ys
	\]
	and so $xe_k=xss^{-1}=yss^{-1}=ye_k$ where $s^{-1}$ is the group inverse of $s$ in $Se_k$. Suppose that $x\in Se_m$ for some $m\in I$. Then 
	\[
	x=xe_m=xe_ke_m=ye_ke_m=ye_m\in[y]
	\] 
	and so $[x]=[y]$. Hence $\rho_s$ is $E^*$-preserving. 
	
	Let $t\in S$. It is clear that $(ts^{-1})\rho_s=ts^{-1}s=te_k\in[t]$. Thus $[t]\cap S\rho_s\neq\emptyset$ for all $[t]\in S/E$ and so $\rho_s\in\reg(T)$. Now, it is concluded that $\rho_s\in Q(\mathbf{2})$, as required.
	
	Finally, define the regular representation $\psi:S\to Q(\mathbf{2})$ by $s\mapsto\rho_s$. It is routine to verify that $\psi$ is a homomorphism. To prove that $\psi$ is injective, assume that $\rho_s=\rho_t$. If $s\in Se_i$ and $t\in Se_j$ for some $i,j\in I$, then
	\[
	s=e_is=e_i\rho_s=e_i\rho_t=e_it=t.
	\]
	Therefore, $\psi$ is a monomorphism.
\end{proof}

Similar to group theory, the semigroup $Q(\mathbf{2})$, which is the union of symmetric groups, plays the same role as the symmetric group in Cayley's Theorem.

\subsection*{Acknowledgments} This research was supported by Chiang Mai University.

\subsection*{Conflict of interest disclosures} The author states no conflict of interest.

\bibliographystyle{abbrv}\addcontentsline{toc}{section}{References}
\bibliography{References}

\begin{flushleft}
\vskip.3in

KRITSADA SANGKHANAN, Research Center in Mathematics and Applied Mathematics, Department of Mathematics, Faculty of Science, Chiang Mai University, Chiang Mai, 50200, Thailand; e-mail: kritsada.s@cmu.ac.th

\end{flushleft}
\end{document}